\documentclass[12pt]{article}
\usepackage{amsmath,amssymb,amsthm}
\numberwithin{equation}{section}
\usepackage{units}
\usepackage{color}
\usepackage[T1]{fontenc}
\usepackage[utf8]{inputenc}
\usepackage{authblk}
\usepackage{bm}
\usepackage{graphicx}

\usepackage{datetime}

\usepackage[colorlinks=true,
linkcolor=webgreen,
filecolor=webbrown,
citecolor=webgreen]{hyperref}

\definecolor{webgreen}{rgb}{0,.5,0}
\definecolor{webbrown}{rgb}{.6,0,0}

\usepackage[toc,page]{appendix}

\newcommand{\C}{{\mathbb C}}

\newcommand{\Z}{{\mathbb Z}}

\newtheorem*{theoremUn}{Theorem}
\newtheorem{thm}{Theorem}

\newtheorem{lemma}{Lemma}
\newtheorem{example}{Example}

\newtheorem{corollary}[thm]{Corollary}

\title{Fibonacci-Zeta infinite series associated with the polygamma functions}

\author[1]{Kunle Adegoke\thanks{adegoke00@gmail.com}}
\affil{Department of Physics and Engineering Physics, \mbox{Obafemi Awolowo University, Ile-Ife, 220005 Nigeria}}
\author[2]{Sourangshu Ghosh\thanks{sourangshug123@gmail.com}}
\affil{Department of Civil Engineering, \mbox{Indian Institute of Technology, Kharagpur}}

\begin{document}
\date{}

\maketitle



\begin{abstract} \noindent
We derive new infinite series involving Fibonacci numbers and Riemann zeta numbers. The calculations are facilitated by evaluating linear combinations of polygamma functions of the same order at certain arguments.
\end{abstract}
\section{Introduction}
The Fibonacci numbers, $F_n$, and the Lucas numbers, $L_n$, are defined, for all integers $n$ by the Binet formulas:
\begin{equation}
F_n  = \frac{{\alpha ^n  - \beta ^n }}{{\alpha  - \beta }},\quad L_n  = \alpha ^n  + \beta ^n\,,
\end{equation}
where $\alpha$ and $\beta$ are the zeros of the characteristic polynomial, $x^2-x-1$, of the Fibonacci sequence. Thus $\alpha+\beta=1$ and $\alpha\beta=-1$; so that $\alpha=(1+\sqrt5)/2$ (the golden ratio) and $\beta=-1/\alpha=(1-\sqrt5)/2$. 
Koshy \cite{koshy} and Vajda \cite{vajda} have written excellent books dealing with Fibonacci and Lucas numbers.

\medskip

Our purpose in writing this paper is to employ the properties of the polygamma functions to derive infinite series identities involving Fibonacci numbers, Lucas numbers and the Riemann zeta numbers. We will obtain sums such as
\[
\sum\limits_{j = 1}^\infty  {\frac{{(j + 1)(j + 2)}}{{3^j }}\zeta (j + 3)F_{2j} }  = \frac{{2\pi ^3 }}{{\sqrt 5 }}\tan \left( {\frac{{\pi \sqrt 5 }}{6}} \right)\sec ^2 \left( {\frac{{\pi \sqrt 5 }}{6}} \right)\,,
\]
and
\[
\sum\limits_{j = 1}^\infty  {( - 1)^{j - 1} j\zeta (j + 1)L_{j - 1} }  = \pi ^2 \sec ^2 \left( {\frac{{\pi \sqrt 5 }}{2}} \right) - 3\,,
\]
and, in fact, more general series. Here $\zeta(n)$ is the Riemann zeta function.

\medskip
The digamma function, $\psi(z)$, is the logarithmic derivative of the Gamma function:
\[
\psi (z) = \frac{d}{{dz}}\log \Gamma (z) = \frac{{\Gamma '(z)}}{{\Gamma (z)}}\,,
\]
where the Gamma function is defined for $\Re(z)>0$ by
\[
\Gamma (z) = \int_0^\infty  {e^{ - t} t^{z - 1}\, dt}  = \int_0^\infty  {\left( {\log (1/t)} \right)^{z - 1}\, dt}\,,
\]
and is extended to the rest of the complex plane, excluding the non-positive integers, by analytic continuation. The Gamma function has a simple pole at each of the points $z=\cdots,-3,-2,-1,0$. The Gamma function extends the classical factorial function to the complex plane: $\Gamma(z)=(z-1)!$ 

\medskip

The nth polygamma function $\psi^{(n)}(z)$ is the nth derivative of the digamma function:
\[
\psi ^{(n)} (z) = \frac{{d^{n + 1} }}{{dz^{n + 1} }}\log \Gamma (z) = \frac{{d^n }}{{dz^n }}\psi ^{(n)} (z),\quad\psi ^{(0)} (z) = \psi (z)\,.
\]

The polygamma functions satisfy the recurrence relation,
\begin{equation}\label{eq.iow3v9r}
\psi ^{(m)} (z + 1) = \psi ^{(m)} (z) + ( - 1)^m \frac{{m!}}{{z^{m + 1} }}\,,
\end{equation}
and the reflection relation,
\begin{equation}\label{eq.ebduq6i}
( - 1)^m \psi ^{(m)} (1 - z) - \psi ^{(m)} (z) = \pi \frac{{d^m }}{{dz^m }}\cot (\pi z)\,.
\end{equation}
The Taylor series for the polygamma functions is
\begin{equation}\label{eq.vj6pc4s}
\sum\limits_{j = 0}^\infty  {( - 1)^{m + j + 1} \frac{{(m + j)!}}{{j!}}\zeta (m + j + 1)z^j }  = \psi ^{(m)} (z + 1),\quad m\ge1\,,
\end{equation}

\begin{equation}\label{eq.dicv5wt}
- \gamma  + \sum\limits_{j = 1}^\infty  {( - 1)^{j + 1} \zeta (j + 1)z^j }  = \psi (z + 1)\,,
\end{equation}
where $\gamma$ is the Euler-Mascheroni constant and $\zeta(n)$, $n\in\C$, is the Riemann zeta function,  defined by
\[
\zeta (n) = \sum_{j = 1}^\infty  {\frac{1}{{j^n }}},\quad\Re(n)>1\,,
\]
and analytically continued to all $n\in\C$ with $\Re(n)>0$, $n\ne1$ through
\[
\zeta (n) = \frac{1}{{1 - 2^{1 - n} }}\sum_{j = 1}^\infty  {\frac{{( - 1)^{j + 1} }}{{j^n }}}\,. 
\]
Further information on the polygamma functions can be found in the books by Erd\'elyi~et~al.~\cite[\S1.16]{erdelyi} and Srivastava and Choi~\cite[p.~33]{srivastava12}. The book by Edwards~\cite{edwards74} is a good treatise on the Riemann zeta function.

\medskip

Infinite series involving Fibonacci numbers and Riemann zeta numbers were also derived by Frontczak~\cite{frontczak20a,frontczak20c,frontczak20d}, Frontczak and Goy~\cite{frontczak20b} and Adegoke~\cite{adegoke20x}.
\section{Preliminary results}
Here we derive more functional equations for the polygamma functions. We also evaluate required linear combinations of the polygamma function at appropriate arguments.
\subsection{Functional equations}
Writing $-z$ for $z$ in the recurrence relation~\eqref{eq.iow3v9r} and making use of the reflection relation~\eqref{eq.ebduq6i}, we obtain the duplication formula
\begin{equation}\label{eq.bbzg3c9}
\psi ^{(m)} ( - z) - ( - 1)^m \psi ^{(m)} (z) = ( - 1)^m \pi \frac{{d^m }}{{dz^m }}\cot (\pi z) + \frac{{m!}}{{z^{m + 1} }}\,,
\end{equation}
and consequently,
\begin{equation}
\begin{split}
\psi ^{(m)} ( - x) - \psi ^{(m)} ( - y) &= \psi ^{(m)} (x) - \psi ^{(m)} (y)   + \frac{{m!}}{{x^{m + 1} }} - \frac{{m!}}{{y^{m + 1} }}\\
&\qquad+ \pi \left. {\frac{{d^m }}{{dz^m }}\cot (\pi z)} \right|_{z = x}  - \pi \left. {\frac{{d^m }}{{dz^m }}\cot (\pi z)} \right|_{z = y},\quad\mbox{$m$ even}\,,
\end{split}
\end{equation}

\begin{equation}
\begin{split}
\psi ^{(m)} ( - x) + \psi ^{(m)} ( - y) &= -(\psi ^{(m)} (x) + \psi ^{(m)} (y))   + \frac{{m!}}{{x^{m + 1} }} + \frac{{m!}}{{y^{m + 1} }}\\
&\qquad- \pi \left. {\frac{{d^m }}{{dz^m }}\cot (\pi z)} \right|_{z = x}  - \pi \left. {\frac{{d^m }}{{dz^m }}\cot (\pi z)} \right|_{z = y},\quad\mbox{$m$ odd}\,.
\end{split}
\end{equation}
In particular, if $x+y=1$, then,
\begin{equation}\label{eq.wosq31s}
\psi ^{(m)} ( - x) - \psi ^{(m)} ( - y) =  - \pi \left. {\frac{{d^m }}{{dz^m }}\cot (\pi z)} \right|_{z = y}  + \frac{{m!}}{{x^{m + 1} }} - \frac{{m!}}{{y^{m + 1} }},\quad\mbox{$m$ even}\,,
\end{equation}

\begin{equation}\label{eq.ed8n3gz}
\psi ^{(m)} ( - x) + \psi ^{(m)} ( - y) =  - \pi \left. {\frac{{d^m }}{{dz^m }}\cot (\pi z)} \right|_{z = y}  + \frac{{m!}}{{x^{m + 1} }} + \frac{{m!}}{{y^{m + 1} }},\quad\mbox{$m$ odd}\,.
\end{equation}
Writing $1/2-z$ for $z$ in~\eqref{eq.ebduq6i} gives
\begin{equation}\label{eq.jbtkz4p}
( - 1)^m \psi ^{(m)} \left( {\frac{1}{2} + z} \right) - \psi ^{(m)} \left( {\frac{1}{2} - z} \right) = \pi \left. {\frac{{d^m }}{{dx^m }}\cot (\pi x)} \right|_{x = 1/2 - z}\,.
\end{equation}
Eliminating $\psi^{(m)}(z)$ between \eqref{eq.iow3v9r} and \eqref{eq.ebduq6i} gives
\begin{equation}\label{eq.owmf0ko}
\psi ^{(m)} (1 + z) - ( - 1)^m \psi ^{(m)} (1 - z) =  - \pi \frac{{d^m }}{{dz^m }}\cot (\pi z) + \frac{{( - 1)^m m!}}{{z^{m + 1} }}\,.
\end{equation}
If $x-y=1$, then the recurrence relation gives
\begin{equation}\label{eq.qx2q4bx}
\psi ^{(m)} (x) - \psi ^{(m)} (y) = \frac{{( - 1)^m m!}}{{y^{m + 1} }}\,,
\end{equation}
while if $x+y=1$, the reflection relation gives
\begin{equation}\label{eq.nep5z0n}
\psi ^{(m)} (x) - ( - 1)^m \psi ^{(m)} (y) =  - \pi \left. {\frac{{d^m }}{{dz^m }}\cot (\pi z)} \right|_{z = x}\,.
\end{equation}
The recurrence relation has the following consequences:
\begin{equation}\label{eq.n0r48hk}
\begin{split}
\psi ^{(m)} (x + 1) - \psi ^{(m)} (y + 1) &= \psi ^{(m)} (x) - \psi ^{(m)} (y)\\
&\qquad- ( - 1)^m m!\frac{{x^{m + 1}  - y^{m + 1} }}{{(xy)^{m + 1} }}\,,
\end{split}
\end{equation}

\begin{equation}\label{eq.a6idmpt}
\begin{split}
\psi ^{(m)} (x + 1) + \psi ^{(m)} (y + 1) &= \psi ^{(m)} (x) + \psi ^{(m)} (y)\\
&\qquad+ ( - 1)^m m!\frac{{x^{m + 1}  + y^{m + 1} }}{{(xy)^{m + 1} }}\,.
\end{split}
\end{equation}
For any $x$ and $y$ such that $1-x\notin\Z^{-}$, $1-y\notin\Z^{-}$ and $x\ne1$, $y\ne1$, relation~\eqref{eq.owmf0ko} implies
\begin{equation}\label{eq.tu58ve4}
\begin{split}
\psi ^{(m)} (1 + x) - \psi ^{(m)} (1 + y) &= ( - 1)^m \left( {\psi ^{(m)} (1 - x) - \psi ^{(m)} (1 - y)} \right)\\
&\quad- \pi \frac{{d^m }}{{dx^m }}\cot (\pi x) + \pi \frac{{d^m }}{{dy^m }}\cot (\pi y)\\
&\qquad+ \frac{{( - 1)^m m!}}{{x^{m + 1} }} - \frac{{( - 1)^m m!}}{{y^{m + 1} }}\,,
\end{split}
\end{equation}

\begin{equation}\label{eq.jbsc0ez}
\begin{split}
\psi ^{(m)} (1 + x) + \psi ^{(m)} (1 + y) &= ( - 1)^m \left( {\psi ^{(m)} (1 - x) + \psi ^{(m)} (1 - y)} \right)\\
&\quad- \pi \frac{{d^m }}{{dx^m }}\cot (\pi x) - \pi \frac{{d^m }}{{dy^m }}\cot (\pi y)\\
&\qquad+ \frac{{( - 1)^m m!}}{{x^{m + 1} }} + \frac{{( - 1)^m m!}}{{y^{m + 1} }}\,.
\end{split}
\end{equation}
From the functional equation~\eqref{eq.tu58ve4}, it follows that if $m$ is even and $x+y=1$, then,
\begin{equation}\label{eq.qvsfpnz}
\psi ^{(m)} (1 + x) - \psi ^{(m)} (1 + y) = \pi \left. {\frac{{d^m }}{{dz^m }}\cot (\pi z)} \right|_{z = y}  + \frac{{m!}}{{x^{m + 1} }} - \frac{{m!}}{{y^{m + 1} }}\,.
\end{equation}
From the functional equation~\eqref{eq.jbsc0ez}, it follows that if $m$ is odd and $x+y=1$, then,
\begin{equation}\label{eq.csjls8y}
\psi ^{(m)} (1 + x) + \psi ^{(m)} (1 + y) = \pi \left. {\frac{{d^m }}{{dz^m }}\cot (\pi z)} \right|_{z = y}  - \frac{{m!}}{{x^{m + 1} }} - \frac{{m!}}{{y^{m + 1} }}\,.
\end{equation}
If $m$ is an even number and $x+y=2$, equation~\eqref{eq.tu58ve4} gives
\begin{equation}\label{eq.i9bj8xc}
\psi ^{(m)} (1 + x) - \psi ^{(m)} (1 + y) =- \pi \left. {\frac{{d^m }}{{dz^m }}\cot (\pi z)} \right|_{z = x}  + \frac{{m!}}{{(1 - y)^{m + 1} }} + \frac{{m!}}{{x^{m + 1} }} - \frac{{m!}}{{y^{m + 1} }}\,,
\end{equation}
while if $m$ is an odd number and $x+y=2$, equation~\eqref{eq.jbsc0ez} gives
\begin{equation}\label{eq.cd0jgko}
\psi ^{(m)} (1 + x) + \psi ^{(m)} (1 + y) =- \pi \left. {\frac{{d^m }}{{dz^m }}\cot (\pi z)} \right|_{z = x}  - \frac{{m!}}{{(1 - y)^{m + 1} }} - \frac{{m!}}{{x^{m + 1} }} - \frac{{m!}}{{y^{m + 1} }}\,.
\end{equation}
Note that we used
\[
\pi \left. {\frac{{d^m }}{{dz^m }}\cot (\pi z)} \right|_{z = y}  + \pi \left. {\frac{{d^m }}{{dz^m }}\cot (\pi z)} \right|_{z = 1 - y}=0,\quad\mbox{$m$ even}\,,
\]
and
\[
- \pi \left. {\frac{{d^m }}{{dz^m }}\cot (\pi z)} \right|_{z = y}  + \pi \left. {\frac{{d^m }}{{dz^m }}\cot (\pi z)} \right|_{z = 1 - y}=0,\quad\mbox{$m$ odd}
\]
From~\eqref{eq.jbtkz4p} we have
\begin{equation}\label{eq.wemv3k4}
\begin{split}
( - 1)^m \left( {\psi ^{(m)} \left( {\frac{1}{2} + x} \right) - \psi ^{(m)} \left( {\frac{1}{2} + y} \right)} \right) &= \psi ^{(m)} \left( {\frac{1}{2} - x} \right) - \psi ^{(m)} \left( {\frac{1}{2} - y} \right)\\
&\qquad + \pi \left. {\frac{{d^m }}{{dz^m }}\cot (\pi z)} \right|_{z = 1/2 - x}  - \pi \left. {\frac{{d^m }}{{dz^m }}\cot (\pi z)} \right|_{z = 1/2 - y}\,, 
\end{split}
\end{equation}

\begin{equation}\label{eq.mqjw8dw}
\begin{split}
( - 1)^m \left( {\psi ^{(m)} \left( {\frac{1}{2} + x} \right) + \psi ^{(m)} \left( {\frac{1}{2} + y} \right)} \right) &= -\psi ^{(m)} \left( {\frac{1}{2} - x} \right) - \psi ^{(m)} \left( {\frac{1}{2} - y} \right)\\
&\qquad + \pi \left. {\frac{{d^m }}{{dz^m }}\cot (\pi z)} \right|_{z = 1/2 - x}  + \pi \left. {\frac{{d^m }}{{dz^m }}\cot (\pi z)} \right|_{z = 1/2 - y}\,. 
\end{split}
\end{equation}
Thus, if $m$ is even and $x+y=1$, then from~\eqref{eq.wemv3k4} and the duplication formula, we have
\begin{equation}\label{eq.ptlkzfr}
\psi ^{(m)} \left( {\frac{1}{2} + x} \right) - \psi ^{(m)} \left( {\frac{1}{2} + y} \right) =  - \pi \left. {\frac{{d^m }}{{dz^m }}\cot (\pi z)} \right|_{z = 1/2 - y}  + \frac{{m!}}{{(x - {1 \mathord{\left/
 {\vphantom {1 2}} \right.
 \kern-\nulldelimiterspace} 2})^{m + 1} }}\,,
\end{equation}
and if $m$ is odd and $x+y=1$, then from~\eqref{eq.mqjw8dw} and the duplication formula, we have
\begin{equation}\label{eq.h46a6f3}
\psi ^{(m)} \left( {\frac{1}{2} + x} \right) + \psi ^{(m)} \left( {\frac{1}{2} + y} \right) =  - \pi \left. {\frac{{d^m }}{{dz^m }}\cot (\pi z)} \right|_{z = 1/2 - y}  - \frac{{m!}}{{(x - {1 \mathord{\left/
 {\vphantom {1 2}} \right.
 \kern-\nulldelimiterspace} 2})^{m + 1} }}\,.
\end{equation}
\subsection{Evaluation at various arguments}
\begin{lemma}
We have
\begin{equation}\label{eq.f078k0u}
\psi ^{(m)} (\alpha ) - \psi ^{(m)} (\beta ) =  - \pi \left. {\frac{{d^m }}{{dz^m }}\cot (\pi z)} \right|_{z = \alpha },\quad\mbox{$m$ even}\,,
\end{equation}

\begin{equation}\label{eq.liqqoev}
\psi ^{(m)} (\alpha ) + \psi ^{(m)} (\beta ) =  - \pi \left. {\frac{{d^m }}{{dz^m }}\cot (\pi z)} \right|_{z = \alpha },\quad\mbox{$m$ odd}\,,
\end{equation}

\begin{equation}\label{eq.t3bna0c}
\psi ^{(m)} (\alpha ^2 ) - \psi ^{(m)} (\beta ^2 ) = \pi \left. {\frac{{d^m }}{{dz^m }}\cot (\pi z)} \right|_{z = \beta }  + m!F_{m + 1} \sqrt 5,\quad\mbox{$m$ even}\,,
\end{equation}

\begin{equation}\label{eq.fn3e46b}
\psi ^{(m)} (\alpha ^2 ) + \psi ^{(m)} (\beta ^2 ) =  - \pi \left. {\frac{{d^m }}{{dz^m }}\cot (\pi z)} \right|_{z = \beta }  - m!L_{m + 1},\quad\mbox{$m$ odd}\,,
\end{equation}

\begin{equation}\label{eq.upo3xfi}
\psi ^{(m)} (\alpha ^3 ) - \psi ^{(m)} (\beta ^3 ) = \pi \left. {\frac{{d^m }}{{dz^m }}\cot (\pi z)} \right|_{z = 2\beta } + \frac{{m!}}{{(\sqrt 5 )^{m + 1} }} + \frac{{m!}}{{2^{m + 1} }}F_{m + 1} \sqrt 5,\quad\mbox{$m$ even}\,, 
\end{equation}

\begin{equation}\label{eq.uunsu5d}
\psi ^{(m)} (\alpha ^3 ) + \psi ^{(m)} (\beta ^3 ) =  - \pi \left. {\frac{{d^m }}{{dz^m }}\cot (\pi z)} \right|_{z = 2\beta } - \frac{{m!}}{{(\sqrt 5 )^{m + 1} }} - \frac{{m!}}{{2^{m + 1} }}L_{m + 1},\quad\mbox{$m$ odd}\,, 
\end{equation}
\begin{equation}\label{eq.e4vm6jo}
\psi ^{(m)} ({{\alpha ^3 } \mathord{\left/
 {\vphantom {{\alpha ^3 } 2}} \right.
 \kern-\nulldelimiterspace} 2}) - \psi ^{(m)} ({{\beta ^3 } \mathord{\left/
 {\vphantom {{\beta ^3 } 2}} \right.
 \kern-\nulldelimiterspace} 2}) =  - \pi \left. {\frac{{d^m }}{{dz^m }}\cot (\pi z)} \right|_{{{z = \sqrt 5 } \mathord{\left/
 {\vphantom {{z = \sqrt 5 } 2}} \right.
 \kern-\nulldelimiterspace} 2}}  + \frac{{m!2^{m + 1} }}{{(\sqrt 5 )^{m + 1} }},\quad\mbox{$m$ even}\,,
\end{equation}

\begin{equation}\label{eq.azxwtdt}
\psi ^{(m)} ({{\alpha ^3 } \mathord{\left/
 {\vphantom {{\alpha ^3 } 2}} \right.
 \kern-\nulldelimiterspace} 2}) + \psi ^{(m)} ({{\beta ^3 } \mathord{\left/
 {\vphantom {{\beta ^3 } 2}} \right.
 \kern-\nulldelimiterspace} 2}) =  - \pi \left. {\frac{{d^m }}{{dz^m }}\cot (\pi z)} \right|_{{{z = \sqrt 5 } \mathord{\left/
 {\vphantom {{z = \sqrt 5 } 2}} \right.
 \kern-\nulldelimiterspace} 2}}  - \frac{{m!2^{m + 1} }}{{(\sqrt 5 )^{m + 1} }},\quad\mbox{$m$ odd}\,,
\end{equation}

\begin{equation}\label{eq.kgq089n}
\psi ^{(m)} ({{\alpha ^r } \mathord{\left/
 {\vphantom {{\alpha ^r } {L_r }}} \right.
 \kern-\nulldelimiterspace} {L_r }}) - \psi ^{(m)} ({{\beta ^r } \mathord{\left/
 {\vphantom {{\beta ^r } {L_r }}} \right.
 \kern-\nulldelimiterspace} {L_r }}) =  -\pi \left. {\frac{{d^m }}{{dz^m }}\cot (\pi z)} \right|_{{{z = \alpha ^r } \mathord{\left/
 {\vphantom {{z = \alpha ^r } {L_r }}} \right.
 \kern-\nulldelimiterspace} {L_r }}},\quad\mbox{$m$ even}\,,
\end{equation}

\begin{equation}\label{eq.n4oblb2}
\psi ^{(m)} ({{\alpha ^r } \mathord{\left/
 {\vphantom {{\alpha ^r } {L_r }}} \right.
 \kern-\nulldelimiterspace} {L_r }}) + \psi ^{(m)} ({{\beta ^r } \mathord{\left/
 {\vphantom {{\beta ^r } {L_r }}} \right.
 \kern-\nulldelimiterspace} {L_r }}) =  - \pi \left. {\frac{{d^m }}{{dz^m }}\cot (\pi z)} \right|_{{{z = \alpha ^r } \mathord{\left/
 {\vphantom {{z = \alpha ^r } {L_r }}} \right.
 \kern-\nulldelimiterspace} {L_r }}},\quad\mbox{$m$ odd}\,,
\end{equation}

\begin{equation}\label{eq.ef1at22}
\psi ^{(m)} \left( {\frac{{\alpha ^r }}{{F_r \sqrt 5 }}} \right) - \psi ^{(m)} \left( {\frac{{\beta ^r }}{{F_r \sqrt 5 }}} \right) = ( - 1)^{rm + r + m} m!F_r^{m + 1} (\alpha ^r \sqrt 5 )^{m + 1}\,,
\end{equation}

\begin{equation}\label{eq.hm02cnz}
\begin{split}
&\psi ^{(m)} (1 + {{2\alpha ^r } \mathord{\left/
 {\vphantom {{2\alpha ^r } {L_r }}} \right.
 \kern-\nulldelimiterspace} {L_r }}) - \psi ^{(m)} (1 + {{2\beta ^r } \mathord{\left/
 {\vphantom {{2\beta ^r } {L_r }}} \right.
 \kern-\nulldelimiterspace} {L_r }})\\
&\qquad=  - \pi \left. {\frac{{d^m }}{{dz^m }}\cot (\pi z)} \right|_{{{z = 2\alpha ^r } \mathord{\left/
 {\vphantom {{z = 2\alpha ^r } {L_r }}} \right.
 \kern-\nulldelimiterspace} {L_r }}}+ \frac{{m!L_r^{m + 1} }}{{(F_r \sqrt 5 )^{m + 1} }} - \frac{{( - 1)^r m!L_r^{m + 1} F_{r(m + 1)} \sqrt 5 }}{{2^{m + 1} }},\quad\mbox{$m$ even}\,,
\end{split}
\end{equation}

\begin{equation}\label{eq.lk3xiqf}
\begin{split}
&\psi ^{(m)} (1 + {{2\alpha ^r } \mathord{\left/
 {\vphantom {{2\alpha ^r } {L_r }}} \right.
 \kern-\nulldelimiterspace} {L_r }}) + \psi ^{(m)} (1 + {{2\beta ^r } \mathord{\left/
 {\vphantom {{2\beta ^r } {L_r }}} \right.
 \kern-\nulldelimiterspace} {L_r }})\\
&\qquad=  - \pi \left. {\frac{{d^m }}{{dz^m }}\cot (\pi z)} \right|_{{{z = 2\alpha ^r } \mathord{\left/
 {\vphantom {{z = 2\alpha ^r } {L_r }}} \right.
 \kern-\nulldelimiterspace} {L_r }}}- \frac{{m!L_r^{m + 1} }}{{(F_r \sqrt 5 )^{m + 1} }} - \frac{{m!L_r^{m + 1} L_{r(m + 1)}}}{{2^{m + 1} }},\quad\mbox{$m$ odd}\,,
\end{split}
\end{equation}

\begin{equation}\label{eq.st4ngz1}
\begin{split}
&\psi ^{(m)} ( - {{\alpha ^r } \mathord{\left/
 {\vphantom {{\alpha ^r } {L_r }}} \right.
 \kern-\nulldelimiterspace} {L_r }}) - \psi ^{(m)} ( - {{\beta ^r } \mathord{\left/
 {\vphantom {{\beta ^r } {L_r }}} \right.
 \kern-\nulldelimiterspace} {L_r }})\\
&\qquad=  - \pi \left. {\frac{{d^m }}{{dz^m }}\cot (\pi z)} \right|_{z = {{\beta ^r } \mathord{\left/
 {\vphantom {{\beta ^r } {L_r }}} \right.
 \kern-\nulldelimiterspace} {L_r }}}  - m!L_r^{m + 1} ( - 1)^r F_{r(m + 1)} \sqrt 5,\quad\mbox{$m$ even}\,,
\end{split}
\end{equation}

\begin{equation}\label{eq.vnyktlk}
\begin{split}
&\psi ^{(m)} ( - {{\alpha ^r } \mathord{\left/
 {\vphantom {{\alpha ^r } {L_r }}} \right.
 \kern-\nulldelimiterspace} {L_r }}) + \psi ^{(m)} ( - {{\beta ^r } \mathord{\left/
 {\vphantom {{\beta ^r } {L_r }}} \right.
 \kern-\nulldelimiterspace} {L_r }})\\
&\qquad=  - \pi \left. {\frac{{d^m }}{{dz^m }}\cot (\pi z)} \right|_{z = {{\beta ^r } \mathord{\left/
 {\vphantom {{\beta ^r } {L_r }}} \right.
 \kern-\nulldelimiterspace} {L_r }}}  + m!L_r^{m + 1} L_{r(m + 1)},\quad\mbox{$m$ odd}\,. 
\end{split}
\end{equation}
\end{lemma}
\begin{proof}
To prove~\eqref{eq.f078k0u} and~\eqref{eq.liqqoev}, set $x=\alpha$, $y=\beta$ in~\eqref{eq.nep5z0n}. To prove~\eqref{eq.t3bna0c} and~\eqref{eq.fn3e46b}, set $x=\alpha$, $y=\beta$ in~\eqref{eq.qvsfpnz} and~\eqref{eq.csjls8y}. Setting $x=2\alpha$, $y=2\beta$ in~\eqref{eq.i9bj8xc} and~\eqref{eq.cd0jgko} gives~\eqref{eq.upo3xfi} and~\eqref{eq.uunsu5d}. To prove~\eqref{eq.e4vm6jo} and~\eqref{eq.azxwtdt} set $x=\alpha$, $y=\beta$ in~\eqref{eq.ptlkzfr} and~\eqref{eq.h46a6f3}. Use of $x=\alpha^r/L_r$ and $y=\beta^r/L_r$ in~\eqref{eq.nep5z0n} gives~\eqref{eq.kgq089n} and~\eqref{eq.n4oblb2}. Identity~\eqref{eq.ef1at22} is obtained by setting $x=\alpha^r/(F_r\sqrt 5)$ and $y=\beta^r/(F_r\sqrt 5)$ in~\eqref{eq.qx2q4bx}. Identities~\eqref{eq.hm02cnz} and~\eqref{eq.lk3xiqf} follow from~\eqref{eq.i9bj8xc} and~\eqref{eq.cd0jgko}, upon setting $x=2\alpha^r/L_r$ and $y=2\beta^r/L_r$. Identities~\eqref{eq.st4ngz1} and~\eqref{eq.vnyktlk} are obtained from~\eqref{eq.wosq31s} and~\eqref{eq.ed8n3gz}, with $x=\alpha^r/L_r$ and $y=\beta^r/L_r$.
\end{proof}
\section{Main results}
\begin{theoremUn}
If $r$ and $m$ are integers, then,
\[\tag{F}
\begin{split}
&\sum\limits_{j = 1}^\infty  {( - 1)^{j + 1} \frac{{(m + j)!}}{{j!}}\zeta (m + j + 1)F_{rj} z^j }\\
&\qquad= \frac{(-1)^m}{{\sqrt 5 }}\left\{ {\psi ^{(m)} (1 + \alpha ^r z) - \psi ^{(m)} (1 + \beta ^r z)} \right\},\quad\mbox{$m\ge0$}\,,
\end{split}
\]

\[\tag{L}
\begin{split}
&\sum\limits_{j = 0}^\infty  {( - 1)^{j} \frac{{(m + j)!}}{{j!}}\zeta (m + j + 1)L_{rj} z^j }\\
&\qquad = (-1)^{m - 1}\left\{\psi ^{(m)} (1 + \alpha ^r z) + \psi ^{(m)} (1 + \beta ^r z)\right\},\quad\mbox{$m\ge1$}\,.
\end{split}
\]
\end{theoremUn}
\begin{proof}
Writing $\alpha^rz$ for $z$ in the Taylor series~\eqref{eq.vj6pc4s}, we obtain
\begin{equation}\label{eq.xowv2u9}
\sum\limits_{j = 0}^\infty  {( - 1)^{j + 1} \frac{{(m + j)!}}{{j!}}\zeta (m + j + 1)\alpha ^{rj} z^j }  = (-1)^m\psi ^{(m)} (\alpha ^r z + 1)\,.
\end{equation}
Similarly,
\begin{equation}\label{eq.br32tgi}
\sum\limits_{j = 0}^\infty  {( - 1)^{j + 1} \frac{{(m + j)!}}{{j!}}\zeta (m + j + 1)\beta ^{rj} z^j }  = (-1)^m\psi ^{(m)} (\beta ^r z + 1)\,.
\end{equation}
Subtraction of~\eqref{eq.br32tgi} from~\eqref{eq.xowv2u9} gives {(F)} while their addition produces~{(L)}, on account of the Binet formulas.
\end{proof}
Note that, in view of identities~\eqref{eq.n0r48hk} and~\eqref{eq.a6idmpt}, the right hand side of {(F)} and {(L)} can be expressed as:
\begin{equation}\label{eq.e6t45xd}
\begin{split}
&\frac{1}{{\sqrt 5 }}\left\{ {\psi ^{(m)} (\alpha ^r z + 1) - \psi ^{(m)} (\beta ^r z + 1)} \right\}\\
&\qquad= \frac{1}{{\sqrt 5 }}\left\{ {\psi ^{(m)} (\alpha ^r z) - \psi ^{(m)} (\beta ^r z)} \right\} - \frac{{( - 1)^r m!}}{{z^{m + 1} }}F_{r(m + 1)},\quad\mbox{$m$ even}\,,
\end{split}
\end{equation}

\begin{equation}\label{eq.d4l10zr}
\begin{split}
&\psi ^{(m)} (\alpha ^r z + 1) + \psi ^{(m)} (\beta ^r z + 1)\\
&\qquad= \psi ^{(m)} (\alpha ^r z) + \psi ^{(m)} (\beta ^r z) - \frac{{m!}}{{z^{m + 1} }}L_{r(m + 1)},\quad\mbox{$m$ odd}\,.
\end{split}
\end{equation}
\begin{corollary}
We have
\begin{equation}\label{eq.nkfpd7u}
\begin{split}
&\sum\limits_{j = 1}^\infty  {( - 1)^{j + 1} \frac{{(m + j)!}}{{j!}}\frac{{\zeta (m + j + 1)}}{{2^j }}F_{3j} }\\
&\qquad=  - \frac{1}{{\sqrt 5 }}\left( {\left. {\pi \frac{{d^m }}{{dz^m }}\cot (\pi z)} \right|_{z = {{\sqrt 5 } \mathord{\left/
{\vphantom {{\sqrt 5 } 2}} \right.
\kern-\nulldelimiterspace} 2}}  - \frac{{m!\,2^{m + 1} }}{{(\sqrt 5 )^{m + 1} }}} \right) + m!\,2^{m + 1} F_{3m + 3},\quad\mbox{$m$ even}\,,
\end{split}
\end{equation}

\begin{equation}\label{eq.me7icdg}
\begin{split}
&\sum\limits_{j = 0}^\infty  {( - 1)^{j} \frac{{(m + j)!}}{{j!}}\frac{{\zeta (m + j + 1)}}{{2^j }}L_{3j} }\\
&\qquad=  -  {\left. {\pi \frac{{d^m }}{{dz^m }}\cot (\pi z)} \right|_{z = {{\sqrt 5 } \mathord{\left/
{\vphantom {{\sqrt 5 } 2}} \right.
\kern-\nulldelimiterspace} 2}}  - \frac{{m!\,2^{m + 1} }}{{(\sqrt 5 )^{m + 1} }}} - m!\,2^{m + 1} L_{3m + 3},\quad\mbox{$m$ odd}\,.
\end{split}
\end{equation}
\end{corollary}
\begin{proof}
Set $r=3$, $z=1/2$ in {(F)} and {(L)}, noting  \eqref{eq.e6t45xd} and \eqref{eq.d4l10zr} and using~\eqref{eq.e4vm6jo} and~\eqref{eq.azxwtdt}.
\end{proof}
\begin{example}
We have
\begin{equation}\label{eq.s1a2ufl}
\begin{split}
&\sum\limits_{j = 1}^\infty  {( - 1)^{j + 1} \frac{{(j + 1)(j + 2)}}{{2^j }}\zeta (j + 3)F_{3j} }\\
&\qquad =  - \frac{{2\pi ^3 }}{{\sqrt 5 }}\cot \left( {\frac{{\pi \sqrt 5 }}{2}} \right)\csc ^2 \left( {\frac{{\pi \sqrt 5 }}{2}} \right) + \frac{{13616}}{{25}}\,,
\end{split}
\end{equation}

\begin{equation}\label{eq.ew630ib}
\sum\limits_{j = 0}^\infty  {( - 1)^j \frac{{(j + 1)}}{{2^j }}\zeta (j + 2)L_{3j} }  = \pi ^2 \csc ^2 \left( {\frac{{\pi \sqrt 5 }}{2}} \right) - \frac{{364}}{5}\,.
\end{equation}
\end{example}
\begin{proof}
To prove~\eqref{eq.s1a2ufl}, set $m=2$ in~\eqref{eq.nkfpd7u}, while $m=1$ in~\eqref{eq.me7icdg} proves~\eqref{eq.ew630ib}.
\end{proof}
\begin{corollary}
If $r$ is an integer, then,
\begin{equation}\label{eq.yx19wd0}
\begin{split}
&\sum\limits_{j = 1}^\infty  {( - 1)^{j + 1} \frac{{(m + j)!}}{{j!\,L_r^j }}\zeta (m + j + 1)F_{rj} }\\
&\qquad = \frac{\pi }{{\sqrt 5 }}\left. {\frac{{d^m }}{{dz^m }}\cot (\pi z)} \right|_{z = {{\beta ^r } \mathord{\left/
 {\vphantom {{\beta ^r } {L_r }}} \right.
 \kern-\nulldelimiterspace} {L_r }}}  - ( - 1)^r m!L_r^{m + 1} F_{r(m + 1)},\quad\mbox{$m$ even}\,,
\end{split}
\end{equation}

\begin{equation}\label{eq.s0dig90}
\begin{split}
&\sum\limits_{j = 0}^\infty  {( - 1)^{j} \frac{{(m + j)!}}{{j!\,L_r^j }}\zeta (m + j + 1)L_{rj} }\\
&\qquad=  - \pi \left. {\frac{{d^m }}{{dz^m }}\cot (\pi z)} \right|_{z = {{\beta ^r } \mathord{\left/
 {\vphantom {{\beta ^r } {L_r }}} \right.
 \kern-\nulldelimiterspace} {L_r }}}  - m!L_r^{m + 1} L_{r(m + 1)},\quad\mbox{$m$ odd}\,.
\end{split}
\end{equation}
\end{corollary}
\begin{proof}
Set $z=1/L_r$ in {(F)} and {(L)}, noting  \eqref{eq.e6t45xd} and \eqref{eq.d4l10zr} and using~\eqref{eq.kgq089n} and~\eqref{eq.n4oblb2}.
\end{proof}
\begin{example}
We have
\begin{equation}\label{eq.s1u6y4q}
\begin{split}
&\sum\limits_{j = 1}^\infty  {( - 1)^{j + 1} \frac{{(j + 1)(j + 2)}}{{3^j }}\zeta (j + 3)F_{2j} }\\ 
&\qquad= \frac{{2\pi ^3 }}{{\sqrt 5 }}\tan \left( {\frac{{\pi \sqrt 5 }}{6}} \right)\sec ^2 \left( {\frac{{\pi \sqrt 5 }}{6}} \right) - 432\,,
\end{split}
\end{equation}

\begin{equation}\label{eq.fc0zaz6}
\sum\limits_{j = 1}^\infty  {( - 1)^{j - 1} j\zeta (j + 1)L_{j - 1} }  = \pi ^2 \sec ^2 \left( {\frac{{\pi \sqrt 5 }}{2}} \right) - 3\,.
\end{equation}
\end{example}
\begin{proof}
Set $m=2$, $r=2$ in~\eqref{eq.yx19wd0} to prove~\eqref{eq.s1u6y4q} and $m=1$, $r=1$ in~\eqref{eq.s0dig90} to prove~\eqref{eq.fc0zaz6}.
\end{proof}
\begin{corollary}
If $r$ is an integer with $|r|>1$, then,
\begin{equation}\label{eq.z19fv58}
\begin{split}
&\sum\limits_{j = 1}^\infty  {( - 1)^{j + 1} \frac{{(m + j)!}}{{j!\,L_r^j }}2^j \zeta (m + j + 1)F_{rj} }\\
&\qquad =  - \frac{1}{{\sqrt 5 }}\left\{ {\left. {\pi \frac{{d^m }}{{dz^m }}\cot (\pi z)} \right|_{z = {{2\alpha ^r } \mathord{\left/
 {\vphantom {{\beta ^r } {L_r }}} \right.
 \kern-\nulldelimiterspace} {L_r }}}  - \frac{{m!L_r^{m + 1} }}{{(F_r \sqrt 5 )^{m + 1} }}} \right\} - ( - 1)^r \frac{{m!L_r^{m + 1} F_{r(m + 1)} }}{{2^{m + 1} }},\quad\mbox{$m$ even}\,,
\end{split}
\end{equation}

\begin{equation}\label{eq.wx8xhdm}
\begin{split}
&\sum\limits_{j = 0}^\infty  {( - 1)^{j} \frac{{(m + j)!}}{{j!L_r^j }}2^j \zeta (m + j + 1)L_{rj} }\\
&\qquad=  - \left. {\pi \frac{{d^m }}{{dz^m }}\cot (\pi z)} \right|_{z = {{2\alpha ^r } \mathord{\left/
 {\vphantom {{2\alpha ^r } {L_r }}} \right.
 \kern-\nulldelimiterspace} {L_r }}}  - \frac{{m!L_r^{m + 1} }}{{(F_r \sqrt 5 )^{m + 1} }} - \frac{{m!L_r^{m + 1} L_{r(m + 1)} }}{{2^{m + 1} }},\quad\mbox{$m$ odd}\,.
\end{split}
\end{equation}
\end{corollary}
\begin{proof}
Set $z=2/L_r$ in {(F)} and {(L)} and use~\eqref{eq.hm02cnz} and~\eqref{eq.lk3xiqf}.
\end{proof}
\begin{example}
We have
\begin{equation}\label{eq.wi3ql4i}
\begin{split}
&\sum\limits_{j = 1}^\infty  {( - 1)^{j + 1} \frac{{(j + 1)(j + 2)}}{{2^j }}\zeta (j + 3)F_{3j} }\\ 
&\qquad= \frac{{13616}}{{25}} - \frac{{2\pi ^3 }}{{\sqrt 5 }}\cot \left( {\frac{{\pi \sqrt 5 }}{2}} \right)\csc ^2 \left( {\frac{{\pi \sqrt 5 }}{2}} \right)\,,
\end{split}
\end{equation}

\begin{equation}\label{eq.allakva}
\sum\limits_{j = 0}^\infty  {( - 1)^j \frac{{(j + 1)}}{{2^j }}\zeta (j + 2)L_{3j} }  = \pi ^2 \csc ^2 \left( {\frac{{\pi \sqrt 5 }}{2}} \right) - \frac{{364}}{5}\,.
\end{equation}
\end{example}
\begin{proof}
Set $m=2$, $r=3$ in~\eqref{eq.z19fv58} to prove~\eqref{eq.wi3ql4i} and $m=1$, $r=3$ in~\eqref{eq.wx8xhdm} to prove~\eqref{eq.allakva}.
\end{proof}
\begin{corollary}
If $r$ is an integer with $|r|>1$, then,
\begin{equation}\label{eq.vxo2yt0}
\sum\limits_{j = 1}^\infty  {\frac{{(m + j)!}}{{j!\,L_r^j }}\zeta (m + j + 1)F_{rj} }= \frac{\pi }{{\sqrt 5 }}\left. {\frac{{d^m }}{{dz^m }}\cot (\pi z)} \right|_{z = {{\beta ^r } \mathord{\left/
 {\vphantom {{\beta ^r } {L_r }}} \right.
 \kern-\nulldelimiterspace} {L_r }}},\quad\mbox{$m$ even}\,,
\end{equation}

\begin{equation}\label{eq.xk99wvg}
\sum\limits_{j = 0}^\infty  {\frac{{(m + j)!}}{{j!\,L_r^j }}\zeta (m + j + 1)L_{rj} }=  - \pi \left. {\frac{{d^m }}{{dz^m }}\cot (\pi z)} \right|_{z = {{\beta ^r } \mathord{\left/
 {\vphantom {{\beta ^r } {L_r }}} \right.
 \kern-\nulldelimiterspace} {L_r }}},\quad\mbox{$m$ odd}\,.
\end{equation}
\end{corollary}
\begin{proof}
Set $z=-1/L_r$ in {(F)} and {(L)} and use~\eqref{eq.st4ngz1} and~\eqref{eq.vnyktlk}.
\end{proof}
\begin{example}
We have
\begin{equation}\label{eq.pnj960x}
\sum\limits_{j = 1}^\infty  {\frac{{(j + 1)(j + 2)}}{{3^j }}\zeta (j + 3)F_{2j} }  = \frac{{2\pi ^3 }}{{\sqrt 5 }}\tan \left( {\frac{{\pi \sqrt 5 }}{6}} \right)\sec ^2 \left( {\frac{{\pi \sqrt 5 }}{6}} \right)\,,
\end{equation}

\begin{equation}\label{eq.gluacxg}
\sum\limits_{j = 0}^\infty  {\frac{{(j + 1)}}{{7^j }}\zeta (j + 2)L_{4j} }  = \pi ^2 \sec ^2 \left( {\frac{{3\pi \sqrt 5 }}{{14}}} \right)\,.
\end{equation}
\end{example}
\begin{proof}
To prove~\eqref{eq.pnj960x}, set $m=2$, $r=2$ in~\eqref{eq.vxo2yt0}. To prove~\eqref{eq.gluacxg}, set $m=1$, $r=4$ in~\eqref{eq.xk99wvg}.
\end{proof}

\hrule

\hrule

\hrule


\begin{thebibliography}{99}

\bibitem{adegoke20x} K.~Adegoke, Fibonacci series from power series, \emph{manuscript},\\ doi: 10.20944/preprints202011.0463.v1 (2020).

\bibitem{edwards74} H. M. Edwards, \emph{Riemann's Zeta Function}, Academic Press, (1974).

\bibitem{erdelyi} A. Erd\'elyi, W. Magnus, F. Oberhettinger and  F. G. Tricomi, \emph{Higher transcendental functions, vol. 1}, Bateman manuscript project, (1981).
\bibitem{frontczak20a} R. Frontczak, Infinite series involving fibonacci numbers and the Riemann zeta function, \emph{Notes on Number Theory and Discrete Mathematics} {\bf 26}:2 (2020), 159--166.

\bibitem{frontczak20c} R.~Frontczak, Problem B-1267, \emph{The Fibonacci Quarterly} {\bf 58}:2 (2020), to appear.

\bibitem{frontczak20d} R.~Frontczak, Problem H-xxx, \emph{The Fibonacci Quarterly} {\bf 58}:2 (2020), to appear.

\bibitem{frontczak20b} R. Frontczak and T. Goy, General infinite series evaluations involving Fibonacci numbers and the Riemann zeta function, \emph{arXiv:2007.14618v1 [math.NT]}, (2020).


\bibitem{koshy} T.~Koshy, \emph{Fibonacci and Lucas Numbers with Applications}, Wiley-Interscience, (2001).


\bibitem{srivastava12} H. M. Srivastava and J. Choi, \emph{Zeta and q-Zeta Functions and Associated Series and Integrals}, Elsevier Inc., (2012).

\bibitem{vajda} S.~Vajda, \emph{Fibonacci and Lucas Numbers, and the Golden Section: Theory and Applications}, Dover Press, (2008).

\end{thebibliography}
\end{document}